\documentclass[11pt]{article}
\usepackage[utf8]{inputenc}
\usepackage[T1]{fontenc}
\usepackage[oldstyle]{fbb}

\usepackage{amssymb}
\usepackage{stmaryrd}
\usepackage{array}
\usepackage{multirow}
\usepackage{amsmath,amsfonts}
\usepackage[tikz]{bclogo}
\usepackage[top=4.34cm]{geometry}
\usepackage{epigraph}

\usepackage{lmodern}
\usepackage{textcomp}
\usepackage{mathtools, bm}
\usepackage{amssymb, bm}
\usepackage[unq]{unique}
\usepackage{enumerate}
\usepackage{comment}
\usepackage[breaklinks]{hyperref}

\usepackage[numbers]{natbib}  

\usepackage[breaklinks]{hyperref}
\usepackage[numbers]{natbib}

\usepackage{amsthm}

\usepackage{cleveref}
\usepackage[utf8]{inputenc}
\usepackage[T1]{fontenc}
\usepackage{enumitem}
\usepackage{tipa}
\usepackage{float}
\usepackage{dirtytalk}
\usepackage{pbox}
\usepackage{xcolor}
\usepackage{setspace}
\usepackage[normalem]{ulem}
\usepackage{geometry}
\usepackage{amsmath}
\usepackage{amsthm}
\usepackage{tikz}
\usepackage{amssymb}
\usepackage{multicol}
\usepackage{microtype}
\usepackage{datetime}
\usepackage{endnotes}
\usepackage{babel}
\usepackage[mathscr]{euscript}
\usepackage{bigfoot} % to allow verbatim in footnote
\newcommand{\personal}[1]{}
\newcommand{\simpleblowup}[2]{#1 [ #2 ]}
\newtheorem{theorem}{Theorem}[section]
\newtheorem*{theorem*}{Theorem}

\newtheorem{lemma}[theorem]{Lemma}

\newtheorem{question}[theorem]{Question}

\DeclareMathOperator{\FR}{FR}

\title{Minimal hypergraph non-jumps}
\author{Benedict Randall Shaw}
\date{June 2025}

\begin{document}

\maketitle

\begin{abstract}
    An \(r\)-uniform hypergraph, or \(r\)-\emph{graph}, has \emph{density} \(|E(G)|/\left|V(G)^{(r)}\right|\). We say \(\alpha\) is a \emph{jump} for \(r\)-graphs if there is some constant \(\delta=\delta(\alpha)\) such that, for each \(\varepsilon>0\) and \(n\geq r\), any sufficiently large \(r\)-graph of density at least \(\varepsilon\) has a subgraph of order \(n\) and density at least \(\alpha+\delta\). For \(r=2\), all \(\alpha\) are jumps. For \(r\geq 3\), Erd\H{o}s showed all \(\left[0,\frac{r!}{r^r}\right)\) are jumps, and conjectured all \([0,1)\) are jumps. Since then, a variety of non-jumps have been proved, using a method introduced by Frankl and Rödl. 
    
Our aim in this paper is to provide a general setting for this method. As an application, we give several new non-jumps, which are smaller than any previously known. We also demonstrate that these are the smallest the current method can prove.
\end{abstract}

\section{Introduction}

For any set \(S\), let \(S^{(r)}\) be the set of all \(r\)-element subsets of \(S\). An \(r\)-\emph{uniform hypergraph}, or \(r\)-\emph{graph}, is an ordered pair \(G=(V(G),E(G))\) such that \(E\subset V^{(r)}\). This notion generalises \emph{graphs}, which are just \(2\)-graphs. We will often abuse notation to identify an \(r\)-graph with its edge set. An \(r\)-graph \(G\) has \emph{order} \(\left|V(G)\right|\), and \emph{density}
\[D(G)=\frac{\left|E(G)\right|}{\left|V(G)^{(r)}\right|},\]
the proportion of possible edges on \(V\) that are present in \(G\). We say \(\alpha\in [0,1)\) is a \emph{jump} for \(r\)-graphs if there is \(\delta(\alpha)>0\) such that for each \(\varepsilon>0\) and \(n\geq r\), any sufficiently large \(r\)-graph of density at least \(\alpha+\varepsilon\) has a subgraph of order \(n\) and density at least \(\alpha+\delta\).

The Erd\H{o}s-Stone-Simonovits theorem \cite{erdos_structure_1946} tells us that for any \(t\geq 2\), given \(\varepsilon>0\) and \(n\), any sufficiently large \(2\)-graph of density at least \(1-\frac1{t-1}+\varepsilon\) contains a complete \(t\)-partite subgraph with vertex classes of size \(n\). In particular, this subgraph has density at least \(1-\frac1t\). Hence every \(\alpha\in[0,1)\) is a jump for \(2\)-graphs, with \(\delta(\alpha)=1-\frac1t-\alpha\) where \(t\) is chosen such that \(1-\frac1{t-1}\leq \alpha <1-\frac1t\).

Erd\H{o}s \cite{erdos_extremal_1971} proved that for any \(r\), every \(\alpha\in\left[0,\frac{r!}{r^r}\right)\) is a jump for \(r\)-graphs. He conjectured that for any \(r\geq 3\), every \(\alpha\in [0,1)\) is a jump for \(r\)-graphs. He offered 1000 dollars for a proof or disproof. In 1984, Frankl and Rödl \cite{frankl_hypergraphs_1984} disproved this conjecture, showing that for any \(r\geq 3\) and \(k>2r\), the density \(\alpha=1-\frac1{k^{r-1}}\) is not a jump for \(r\)-graphs.

Erd\H{o}s also conjectured that \(\frac{r!}{r^r}\) should be a jump for \(r\)-graphs, offering 500 dollars for a proof or disproof. This remains open, and motivates interest in finding new non-jumps. In 2007, Frankl, Peng, Rödl, and Talbot \cite{frankl_note_2007} showed that for any \(r\geq 3\), the density \(\frac52\cdot\frac{r!}{r^r}\) is a non-jump for \(r\)-graphs. In 2009, Peng \cite{peng_jumping_2009} generalised this by giving a relation between different values of \(r\):

\begin{theorem}[\cite{peng_jumping_2009}] For \(r\leq s\), if \(\alpha\frac{r!}{r^r}\) is a non-jump for \(r\)-graphs, then \(\alpha\frac{s!}{s^s}\) is a non-jump for \(s\)-graphs.\label{thm:increaser}
\end{theorem}
In 2021, Yan and Peng \cite{yan_non-jumping_2023} gave a smaller non-jump, showing that, the density \(\frac{12}{25}\) is a non-jump for \(3\)-graphs. Hence, for each \(r\geq 3\), the density \(\frac{54}{25}\cdot\frac{r!}{r^r}\) is a non-jump for \(r\)-graphs.  For other non-jump results, see \cite{gu_non-jumping_2018,peng_non-jumping_2007,peng_subgraph_2007,peng_using_2007,peng_using_2008,peng_jumping_2009-1,peng_generating_2008}.

All these results have essentially used the same method, which we will refer to as the \emph{Frankl-Rödl method}. We give a general setting for these results. As our main application, we find smaller non-jumps for \(r\geq 3\) than any previously known, as follows:
\begin{theorem}
    For \(r\geq 4\), the density \(2\cdot\frac{r!}{r^r}\) is a non-jump for \(r\)-graphs.\label{thm:alphastar4}
\end{theorem}
\begin{theorem}
    The density \(\frac6{121}\left(5\sqrt5-2\right)= 0.4552\dots\) is a non-jump for \(3\)-graphs.\label{thm:alphastar3}
\end{theorem}
Henceforth, we call these non-jumping densities \(\alpha_r^*\) and \(\alpha_3^*\) respectively. Note that \(\alpha_3^*\) is slightly larger than \(\frac49\), the value of \(2\frac{r!}{r^r}\) at \(r=3\). We will show that in fact these non-jumps are the smallest that can be found using the Frankl-Rödl method.

We mention that there have also been some results to find jumps using computational methods, using Razborov's method of flag algebras \cite{razborov_3-hypergraphs_2010} to bound Turán densities. Baber and Talbot \cite{baber_hypergraphs_2011} used this to show that every \(\alpha\in [0.2299,0.2315]\cup\left[0.2871,\frac8{27}\right)\) is a jump for 3-graphs. They also noted that if \(\frac29=\frac{3!}{3^3}\) is a jump for \(3\)-graphs, then it is a very small jump, with \(\delta<10^{-4}\).

There has also been interest in the question of jumps for multigraphs rather than hypergraphs. We define \(q\)-multigraphs as multigraphs with edge multiplicities bounded by \(q\), and define the density of a multigraph \(G\) of order \(n\) as \[D(G)=\frac{\left|E(G)\right|}{{n\choose 2}}.\] Brown, Erd\H{o}s, and Simonovits \cite{brown_algorithmic_1985} proved that every density \(\alpha\in [0,2)\) is a jump for \(2\)-multigraphs, and conjectured that for any \(q\), every \(\alpha\in [0,q)\) is a jump for \(q\)-multigraphs. However, Rödl and Sidorenko \cite{rodl_jumping_1995} disproved this for \(q\geq 4\), showing that each integer \(3,4,\dots,q-1\) is a non-jump for \(q\)-multigraphs. Horn, La Fleur, and Rödl \cite{horn_jumps_2013} later proved that for \(q\geq 3\), every \(\alpha\in[0,2)\) is a jump for \(q\)-multigraphs, but that for any rational \(r\), the density \(q-r\) is a non-jump for \(q\)-multigraphs for sufficiently large \(q\). The question of whether every \(\alpha\in [0,3)\) is a jump for \(3\)-multigraphs remains open.

\,

The plan of the paper is as follows. In Section 2, we will define a number of notions which will be used throughout this paper, such as Tur\'an density, blow-ups, and Lagrangians. We prove a necessary and sufficient criterion in terms of the Tur\'an density and Lagrangians of a family \(\mathcal{F}\) for \(\alpha\) to be a jump. We generalise our notion of \(r\)-graphs to \(r\)-\emph{patterns}, and extend the definition of the Lagrangian, as well as several minor results, to this new setting.

In Section 3, we state the Frankl-R\"odl method in terms of the Lagrangians of certain \(r\)-patterns. We prove a sufficient condition for \(\alpha\) to be a non-jump, and use a random construction to meet it. This section is largely a restatement of Frankl and R\"odl's original proof, but in terms which we hope will motivate the later argument.

In Section 4, we give proofs of our new non-jumps, which reduce to standard polynomial optimisation problems. From our framework, we also demonstrate that this argument, in its current form, cannot give any smaller non-jumps.

\section{Turán densities, blow-ups, Lagrangians, and patterns}

\subsection{Turán density}

Where \(\mathcal{F}\) is a family of \(r\)-graphs, we say that an \(r\)-graph \(G\) is \emph{\(\mathcal{F}\)-free} if it contains no subgraph isomorphic to an \(r\)-graph in \(\mathcal{F}\). The \emph{Turán density} of \(\mathcal{F}\) is then
\[\pi\left(\mathcal{F}\right)=\lim_{n\to\infty}\max\left\{D(G): G\text{ is an }\mathcal{F}\text{-free }r\text{-graph of order }n\right\}.\]
This exists by a simple averaging argument. We will often write \(\pi(H)\) to mean \(\pi\left(\left\{H\right\}\right)\). Note that given a family of \(r\)-graphs \(\mathcal{F}\) and some \(\varepsilon>0\), any sufficiently large graph of density at least \(\pi\left(\mathcal{F}\right)+\varepsilon\) must contain a member of \(\mathcal{F}\) as a subgraph.

\subsection{Blow-ups}

Given an \(r\)-graph \(G\), and positive integers \(k_v: v\in V(G)\), we define the \emph{blow-up} \(G\left(\mathbf{k}\right)\) as the \(r\)-graph with vertices of the form \((v,i)\), where \(v\in V(G)\) and \(1\leq i \leq k_v\), and where the edge \((v_1,i_1)\dots(v_r,i_r)\) is present if and only if \(v_1\dots v_r\) is an edge of \(G\). We will write \(G(k)\) to mean the blow-up \(G(\mathbf{k})\) in which \(k_v=k\) for each \(v\).

For example, the blow-ups of the \(r\)-graph of a single edge are exactly the \emph{complete \(r\)-partite \(r\)-graphs}, those \(r\)-graphs \(G\) whose vertex set can be partitioned into \(r\) classes such that \(E(G)\) is exactly those edges of \(V(G)^{(r)}\) with one vertex from each class. Notice that for any \(r\)-graph \(G\), a blow-up of a blow-up of \(G\) is itself a blow-up of \(G\). Likewise, an induced subgraph of a blow-up of \(G\) is itself a blow-up of an induced subgraph of \(G\).

We will need the following folklore result:

\begin{theorem}[Blowing up]\label{thm:blow-up}
    Given a family of \(r\)-graphs \(\mathcal{F}\), and some blowup \(H(\mathbf{k})\) of an \(r\)-graph \(H\in\mathcal{F}\),
    \[\pi\left((\mathcal{F}\setminus H) \cup H(\mathbf{k})\right)=\pi\left(\mathcal{F}\right)\]
\end{theorem}

For a proof, see \cite{chapman_hypergraph_2011}. Note that the graph of a single edge has Turán density \(0\). Hence for any \(k\), the complete \(r\)-partite \(r\)-graph with \(k\) vertices in each class also has Turán density \(0\). But these graphs have density at least \(\frac{r!}{r^r}\), implying Erd\H{o}s' result that every \(\alpha\in \left[0,\frac{r!}{r^r}\right)\) is a jump for \(r\)-graphs.

\subsection{Lagrangians}
Given an \(r\)-graph \(G=(V,E)\), a \emph{weighting} \(w\) of \(V\) assigns to each vertex \(v\) a non-negative weight \(w(v)\) such that the total weight \(\sum_{v\in V}w(v)\) is \(1\). We say an edge \(e\) has weight
\[w(e)=\prod_{v\in e}w(v),\]
and that the entire \(r\)-graph has weight
\[w(G)=\sum_{e\in G}w(e).\]
Now the \emph{Lagrangian} of \(G\) is
\[\lambda(G)=\max\{w(G):w\text{ is a weighting of }V(G)\}.\]
This is attained by a compactness argument. We call a weighting achieving this maximum a \textit{maximal weighting of \(G\)}. For another hypergraph \(H\) with \(V(H)\subseteq V(G)\), \(w(H)\) is defined in the same way, even though the total weight on \(V(H)\) may now be less than one. Notice that whenever \(H\subseteq G\), \(w(H)\leq w(G)\). 

The Lagrangian was first introduced in 1965 by Motzkin and Straus \cite{motzkin_maxima_1965}, who characterised it completely for \(2\)-graphs:

\begin{theorem}[Motzkin and Straus \cite{motzkin_maxima_1965}]
	Let \(G\) be a graph, and let \(K_t\) be the largest complete subgraph of \(G\). Then
	\[\lambda(G)=\lambda(K_t)=\frac12\left(1-\frac1t\right).\]
\end{theorem}
This is attained by setting \(w(v)=\frac{1}{t}\) for each vertex \(v\) of the subgraph \(K_t\), and \(w(v)=0\) otherwise. By considering setting some weights to zero, we remark that for any \(r\)-graph \(G\) and subgraph \(H\) of \(G\), we have \(\lambda(H)\leq \lambda(G)\).

We have the following crucial characterisation of the Lagrangian, observed by Frankl and Rödl \cite{frankl_hypergraphs_1984}. For completeness, we include a proof.
\begin{lemma}
    For any \(r\)-graph \(G\), as \(n\to\infty\), \label{lem:blowupdensity}
    \[\max\left\{D\left(G(\mathbf{k})\right): G(\mathbf{k})\text{ has order }n\right\}\to r!\lambda(G).\]
\end{lemma}
\begin{proof}
	This limit exists by an averaging argument. Consider for such a \(\mathbf{k}\) the weighting \(w_\mathbf{k}\) of \(G\) defined by \(w_\mathbf{k}(v)=k_v/n\). Then \[\left|G(\mathbf{k})\right|=\sum_{e\in H}\prod_{v\in e}k_v=n^r w_\mathbf{k}(G)\leq r!\lambda(G)+o(1).\]
	Likewise, for \(w\) a maximal weighting, by choosing each \(\mathbf{k}\) such that \(\frac{k_v}n\) approaches \(w(v)\) as \(n\to\infty\), we find blowups of density \(D(G(\mathbf{k}))\to r!\lambda(G)\), giving the desired limit.
\end{proof}

In particular, this implies that for any \(r\)-graph \(G\) and blow-up \(G(\mathbf{k})\) of \(G\), we have \(\lambda(G(\mathbf{k}))=\lambda(G)\).

Suppose that some \(r\)-graph \(G\) has \(\pi(G)\leq\alpha\), and \(\alpha < r!\lambda(G)\). Then by Theorem \ref{thm:blow-up}, each blow-up \(G(\mathbf{k})\) also has Tur\'an density \(\pi(G)\leq \alpha\). But then Lemma \ref{lem:blowupdensity} implies that sufficiently large graphs of density at least \(\alpha\) must contain large \(G(\mathbf{k})\). We can choose these to have density approaching \(r!\lambda(G)\), and so \(\alpha\) is a jump. In fact, Frankl and Rödl \cite{frankl_hypergraphs_1984} showed that we can characterise all jumps in this way. We again provide a proof, to keep this paper self-contained.

\begin{lemma}\label{lem:jumpthreshold}
		Let \(\alpha\in[0,1)\). Then \(\alpha\) is a jump if and only if there is a finite family of \(r\)-graphs \(\mathcal{F}\) such that \(\pi(\mathcal{F})\leq\alpha\), and \(\alpha<r!\lambda(H)\) for each \(H\in\mathcal{F}\).
	\end{lemma}
	\begin{proof}
		We first show the `if' direction. Choose \(\delta>0\) such that \(\alpha+\delta<r!\lambda(H)\) for each \(H\in\mathcal{F}\). Suppose \(\varepsilon>0\) and \(n'\geq r\). Now Lemma \ref{lem:blowupdensity} implies that for each \(H\), some blow-up \(H(\mathbf{k}^{(H)})\) with order at least \(n'\) has density at least \(\alpha+\delta\). Then define 
		\[\mathcal{F}'=\left\{H(\mathbf{k}^{(H)}):H\in\mathcal{F}\right\}.\]
		Repeated application of Theorem \ref{thm:blow-up} shows that \(\pi(\mathcal{F}')=\pi(\mathcal{F})<\alpha+\varepsilon\). Thus sufficiently large \(r\)-graphs of density at least \(\alpha+\varepsilon\) must contain a member of \(\mathcal{F}'\), and so \(\alpha\) is a jump.
		
		We now show the `only if' direction. Take some \(\delta=\delta(\alpha)\) as in the definition of a jump, and let \(\mathcal{F}\) be the family of all \(r\)-graphs with order \(n\) and density at least \(\alpha+\delta\), for some fixed \(n\) such that
		\[\frac{{n\choose r}}{n^r}\left(\alpha+\delta\right)>\frac{\alpha}{r!}.\]
		Now we have \(\pi(\mathcal{F})\leq\alpha\). But assigning weight \(\frac1n\) to each vertex, we find that all \(H\in\mathcal{F}\) have \(\lambda(H)>\frac{\alpha}{r!}\), giving the desired result.
	\end{proof}

    For example, each \(\alpha\in \left[0,\frac{r!}{r^r}\right)\) is a jump because \(K_r^{(r)}\), the graph of a single edge, has Turán density \(0\) and Lagrangian \(r^{-r}\).

\subsection{Patterns}

Our \(r\)-graphs are \emph{simple}, in that each edge contains each vertex at most once. Hou, Li, Yang, and Zhang \cite{hou_generating_2024} introduced the notion of \emph{\(r\)-patterns}, defined in the same way as \(r\)-graphs except that we now allow an edge to be any multiset of size \(r\) on \(V\).

Given an \(r\)-pattern \(P\), and positive integers \(k_v: v\in V(P)\), we define the blow-up just as we did for \(r\)-graphs: we say \(P(\mathbf{k})\) has vertex set \(\{(v,i):v\in V(P), 1\leq i \leq k_v\}\), and the multiset \((v_1,i_1)\dots(v_r,i_r)\) is present as an edge if \(v_1\dots v_r\in E(P)\). We define the \emph{simple blowup} \(\simpleblowup{P}{\mathbf{k}}\) to be the \(r\)-graph on the same vertex set containing exactly the edges of \(P(\mathbf{k})\) which do not repeat any vertices. That is to say, \(E\left(\simpleblowup{P}{\mathbf{k}}\right)=E(P(\mathbf{k}))\cap V^{(r)}\). Again we write \(P(k)\), \(\simpleblowup{P}{k}\) for the blow-up where \(k_v=k\) for all \(v\).

In order to preserve a version of Lemma \ref{lem:blowupdensity}, we will need a slightly subtle definition of the Lagrangian for \(r\)-patterns, due to Hou, Li, Yang, and Zhang \cite{hou_generating_2024}. For an edge \(e\) of an \(r\)-pattern \(P\) in which each vertex \(v\) has multiplicity \(m_e(v)\), and a weighting \(w\) of the vertices of \(P\), we define
\[w(e)=\prod_{v\in V(P)}\frac{w(v)^{m_e(v)}}{m_e(v)!}.\]
Having defined the weight of edges, we define \(w(P)\) and \(\lambda(P)\) as before. By similar analysis to Lemma \ref{lem:blowupdensity}, we find that for any \(r\)-pattern \(P\), \[\max\left\{D\left(\simpleblowup{P}{\mathbf{k}})\right): \simpleblowup{P}{\mathbf{k}}\text{ has order }n\right\}\to r!\lambda(P).\]
In particular, again we notice that for each \(r\)-pattern \(P\) and \(\mathbf{k}\), we have \(\lambda(\simpleblowup{P}{\mathbf{k}})\leq \lambda(P(\mathbf{k}))=\lambda(P)\).

\section{The Frankl-R\"odl method}

We now introduce Frankl and R\"odl's method \cite{frankl_hypergraphs_1984} for proving non-jumps. Given an \(r\)-pattern \(P\) and \(v\in V(P)\), we define the \emph{Frankl-R\"odl construction} \(\FR_v(P)\) from \(P\) as follows:
\begin{itemize}
    \item Set \(k_v=r\), and \(k_u=1\) for all other vertices \(u\).
    \item Let \(P'\) be the \(r\)-pattern on the vertices of \(P(\mathbf{k})\) with
    \[E(P')=\left\{e\in E(P(\mathbf{k})): \forall i\in\{2,\dots,r\}, m_e((v,i))\leq 1\right\}.\]
    \item Then we define \(\FR_v(P)\) to be \(P'\cup\left\{(v,1),\dots,(v,r)\right\}\).
\end{itemize}

The Frankl-R\"odl method proves the following result:

\begin{theorem}\label{thm:nonjump}
		Given an \(r\)-pattern \(P\), let \(v\) be a vertex of \(P\) such that some maximal weighting of \(P\) assigns \(v\) positive weight. Suppose that the Frankl-R\"odl construction \(\FR_v(P)\) has
        \[\lambda\left(\FR_v(P)\right)=\lambda(P)<1.\]
		Then \(\lambda(P)\) is not a jump for \(r\)-graphs.
	\end{theorem}

    \begin{table}
\centering
    \begin{tabular}{|m{7em}|c|m{0.31\textwidth}|m{0.35\textwidth}|}\hline
    authors&\(r\)&\(\alpha\)&edges of \(P\)\\\hline\hline
    Frankl, Rödl \cite{frankl_hypergraphs_1984}&\(\geq 3\)&\(1-\frac1{k^{r-1}}\) for \(k>2r\)&all edges on \([k]\) not of the form \(vv\dots v\)\\\hline
    \multirow{3}{7em}{Frankl, Peng, Rödl, Talbot \cite{frankl_note_2007}}&\(3\)&\(\frac59\)&\(112,133,123,223\)\\
    &\(3\)&\(1-\frac3l+\frac{3s+2}{l^2}\) for\newline \(s\geq1\), \(l\geq 9s+6\)&\([l]^{(3)}\), and each \(ijj\) with \(j-i\mod l\in\{1,\dots,s\}\)\\\hline
    \multirow{3}{7em}{Yan, Peng \cite{yan_non-jumping_2023}}&\(3\)&\(\frac{12}{25}\)&\(112,123,223\)\\
    &\(3\)&\(\frac{2k-6k^3+4k^4-k\sqrt{4k-1}+4k^2\sqrt{4k-1}}{(2k^2+1)^2}\) for \(k\geq 2\)&\([2k+1]^{(3)}\), and each edge \(e\) with \(m_e(2k+1)=2\)\\\hline
    \end{tabular}
    \caption{Selected known non-jumps and patterns used to prove them. Results which can be derived from others in the table using Theorem \ref{thm:increaser} have been omitted.}\label{table:known}
\end{table}

All currently proven non-jumps known to us can be expressed in the language of Theorem \ref{thm:nonjump}. For more detail on selected results, see Table \ref{table:known}. Towards this theorem, we first prove a sufficient condition for a density to be a non-jump:

\begin{lemma}\label{lem:sufficientcondition}
    Given \(\alpha,r\), suppose that for each integer \(m\), there exists an \(r\)-graph \(G\) with \[\alpha<r!\lambda(G)\]
    such that any subgraph \(H\) of \(G\) on at most \(m\) vertices has \(r!\lambda(H)\leq \alpha\). Then \(\alpha\) is a non-jump for \(r\)-graphs.
\end{lemma}
\begin{proof}
    Suppose \(\alpha\) is a jump for \(r\)-graphs. Then by Lemma \ref{lem:jumpthreshold}, there is a finite family of \(r\)-graphs  \(\mathcal{F}\) with \(\pi(\mathcal{F})\leq\alpha\), and \(\alpha<r!\lambda(H)\) for each \(H\in\mathcal{F}\). Fix \(m=\max\{|V(H)|: H\in\mathcal{F}\}\), and consider the corresponding \(r\)-graph \(G\). Let \(r!\lambda(G)=\alpha+\delta\). Then we can find arbitrarily large blow-ups of \(G\) with density at least \(\alpha+\frac12\delta\). But \(\pi(\mathcal{F})\leq\alpha\) implies that some \(H\in\mathcal{F}\) is contained in some blow-up of \(G\). Take \(H'\) to be this blow-up's copy of \(H\), and let \[S=\{v:(v,i)\in H'\}\subset V(G).\] Now \(H'\) is a subgraph of a blow-up of \(G[S]\). But \(S\) contains at most \(m\) vertices, so
    \[\alpha < r!\lambda(H)\leq r!\lambda(G[S])\leq\alpha,\]
    a contradiction. Thus \(\alpha\) must be a non-jump for \(r\)-graphs.
\end{proof}

    To construct graphs meeting the conditions of Lemma \ref{lem:sufficientcondition}, we will need the following technical construction:
    
	\begin{lemma}\label{lem:probabilistic}
		Suppose \(m,r,c\) are fixed, with \(m,r\) positive integers. Then for sufficiently large \(t\), there exists an \(r\)-graph \(A=A^{(r)}_{m,c,t}\) of order \(t\), with at least \(ct^{r-1}\) edges, such that any subgraph \(H\) of \(A\) with \(r\leq v(H) \leq m\) has at most \(v(H)-r+1\) edges.
	\end{lemma}
	\begin{proof}
		Write \(p=3\frac{c}{t}r!\), and let \(A^*\) be a random graph distributed as \(G(t,p)\), i.e.\ where the vertex set is \([t]\), and each edge in \([n]^{(r)}\) is present with probability \(p\) independently. Now for sufficient \(t\), the expected number of edges in \(A^*\) is at least \(2ct^{r-1}\). We will delete edges from \(A^*\) to obtain \(A\).
		
		Notice that it suffices to check the condition on \(A\) holds for \(H\) an induced subgraph. For any \(S\in [t]^{(s)}\) with \(r\leq s \leq m\), call \(G[S]\) \textit{bad} if it has at least \(s-r+2\) edges. The probability \(G[S]\) is bad is at most \[{{\left|S^{(r)}\right|}\choose{s-r+2}}p^{s-r+2}=O\left(\frac1{t^{s-r+2}}\right)\]
		as \(t\to\infty\). Thus the expected number of bad subgraphs is at most
		\[\sum_{s=r}^m {t\choose s}O\left(\frac1{t^{s-r+2}}\right)=O(t^{r-2}).\]
		
		Define \(A\) from \(A^*\) by deleting all the edges of any bad \(G[S]\). Now the expected number of edges deleted is \(O(t^{r-2})\), so the expected number of edges remaining is still at least \(ct^{r-1}\). Then some choice of \(A\) has at least this many edges, and no bad subgraphs. Thus we have the desired result.
	\end{proof}

    We are now ready to prove Theorem \ref{thm:nonjump}.

    \begin{proof}[Proof of Theorem \ref{thm:nonjump}]
    We show that for any \(m\), we can construct an \(r\)-graph \(G\) that meets the conditions of Lemma \ref{lem:sufficientcondition}. Let \(w\) be a maximal weighting of \(P\) with \(w(v)>0\), and define a weighting \(w'\) of \(G'=\simpleblowup{P}{t}\) by \(w'((v,i))=\frac1tw(v)\). Notice that each edge \(e\) of \(P\) corresponds to 
    \[\prod_{v\in e}{t\choose{m_e(v)}}=t^r\prod_{v\in e}\left(\frac1{m_e(v)!}-O\left(\frac1t\right)\right)\]
    edges in the simple blow-up \(G'\). In particular,
    \[w'(G')=w(P)-O\left(\frac1t\right).\]
    As \(r!\lambda(P)<1\), the edge \(vv\dots v\) is not in \(P\). Define \(G\) from \(G'\) by adding a copy of \(A=A^{(r)}_{m,c,t}\) on the vertex set \(\{(v,1),\dots,(v,t)\}\). Notice that
    \[w'(G)\geq w'(G')+ct^{(r-1)}\left(\frac{w(v)}{t}\right)^r=w'(G')+\frac{c}{t}w(v)^r.\]
    But now we may choose \(c\) so that for sufficiently large \(t\), we have \[w'(G)\geq w(P)+\frac1t=\alpha+\frac1t.\]
    In particular, \(\lambda(G)>\alpha\).

    Now consider a subgraph \(H\) of \(G\) on at most \(m\) vertices. Notice that \(H\) is the union of \(H'=H\cap G'\) and \(H^*=H \cap A\). Let \(w^*\) be a maximal weighting of \(H\), and write \(V(H^*)=\{v_1,\dots,v_s\}\) where \(w^*(v_1)\geq\dots\geq w^*(v_s)\). Now for each \(r\leq i\leq s\), there are at most \(i-r+1\) edges of \(H\) in \(\{v_1,\dots,v_i\}^{(r)}\). In particular, at most \(i-r+1\) edges have weight at least \(w(v_1)\cdots w(v_{r-1}) w(v_i)\). Thus
	\[w(H^*)\leq \sum_{i=r}^sw(v_1)\cdots w(v_{r-1}) w(v_i).\]
    Consider the graph \(H_0=H'\cup\{v_1v_2\dots v_{r-1}v_i: r\leq i \leq s\}\). Now we must have \(w(H)\leq w(H_0)\). But \(H_0\) is a subgraph of a simple blowup of \(\FR_v(P)\), where \(v_1,\dots,v_{r-1}\in H_0\) are blow-ups of \((v,2)\dots,(v,r)\in \FR_v(P)\), and \(\{v_r,\dots,v_s\}\) is the blow-up of \((v,1)\). But then
    \[\lambda(H)\leq\lambda(H_0)\leq \lambda(\FR_v(P))=\alpha.\]
    Thus we have met the conditions of Lemma \ref{lem:sufficientcondition}, and \(\alpha\) is a non-jump for \(r\)-graphs.
    \end{proof}

	\section{Minimal non-jumps}

    We first
		
	\begin{theorem}\label{thm:new1}
		Let \(\alpha=\frac6{121}(5\sqrt{5}-2)\approx0.4552\). Then \(\alpha\) is not a jump for \(3\)-graphs.
	\end{theorem}
	\begin{proof}
		We apply Theorem \ref{thm:nonjump} to
		\[P=\left\{122,123,133,134,144,234\right\},\]
        taking \(v=2\). We illustrate this in Figure \ref{fig:n1}.
		 
		 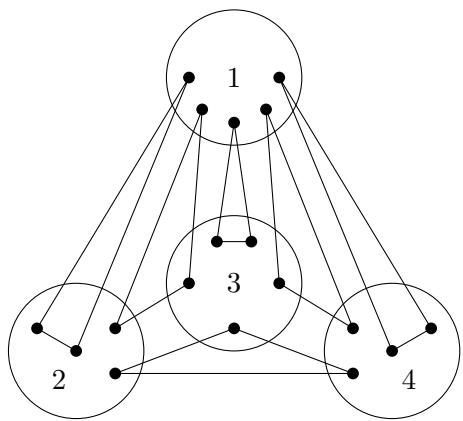
\begin{figure}[h]
		 	\centering
		 	\begin{tikzpicture}[scale=0.6]
		 	\draw (0,6.062177) circle (1.5cm);
		 	\draw (3.5,0) circle (1.5cm);
		 	\draw (-3.5,0) circle (1.5cm);
		 	\draw (0,1.5) circle (1.5cm);
		 	
		 	\fill (1,6.062177) circle (0.13cm);
		 	\fill (-1,6.062177) circle (0.13cm);
		 	\fill (0,5.062177) circle (0.13cm);
		 	\fill (0.707106,5.355071) circle (0.13cm);
		 	\fill (-0.707106,5.355071) circle (0.13cm);
		 	
		 	\fill (0,0.5) circle (0.13cm);
		 	\fill (0.382683,2.423879) circle (0.13cm);
		 	\fill (-0.382683,2.423879) circle (0.13cm);
		 	\fill (1,1.5) circle (0.13cm);
		 	\fill (-1,1.5) circle (0.13cm);
		 	
		 	\fill (2.633974,-0.5) circle (0.13cm);
		 	\fill (2.633974,0.5) circle (0.13cm);
		 	\fill (3.5,0) circle (0.13cm);
		 	\fill (4.366025,0.5) circle (0.13cm);
		 	
		 	\fill (-2.633974,-0.5) circle (0.13cm);
		 	\fill (-2.633974,0.5) circle (0.13cm);
		 	\fill (-3.5,0) circle (0.13cm);
		 	\fill (-4.366025,0.5) circle (0.13cm);
		 	
		 	\draw (1,6.062177) -- (3.5,0) -- (4.366025,0.5) -- (1,6.062177);
		 	\draw (-1,6.062177) -- (-3.5,0) -- (-4.366025,0.5) -- (-1,6.062177);
		 	\draw (0,0.5) -- (2.633974,-0.5) -- (-2.633974,-0.5) -- (0,0.5);
		 	\draw (0.707106,5.355071) -- (1,1.5) -- (2.633974,0.5) -- (0.707106,5.355071);
		 	\draw (-0.707106,5.355071) -- (-1,1.5) -- (-2.633974,0.5) -- (-0.707106,5.355071);
		 	\draw (0,5.062177) -- (0.382683,2.423879) -- (-0.382683,2.423879) -- (0,5.062177);
		 	
		 	\path (0,6.062177) node {$1$} (-3.875,-0.649519) node {$2$} (0,1.5) node {$3$} (3.875,-0.649519) node {$4$};
		 	\end{tikzpicture}
		 	\caption{The pattern used to prove Theorem \ref{thm:new1}.\label{fig:n1}}
		 \end{figure}
		
		Let \(w\) be a maximal weighting of \(P\), and write \(w_v=w(v)\) for convenience. Now
		\[w(P)={w_1\left(\frac{w_2^2+w_3^2+w_4^2}{2}+w_3(w_2+w_4)\right)+w_2w_3w_4}.\]
		
		We now maximise \(w(P)\). If any \(w_i=0\), this expression reduces to either \(w_2w_3w_4\), or to something that is at most \(w_1\frac{(1-w_1)^2}{2}\), which respectively have maxima \(\frac1{27},\frac2{27}\). Having checked these, we can restrict our search to maxima with all \(w_i\) positive. At any local maximum in this region, the partial derivatives in \(w_i\) of \(w(P)\) are all equal. Thus
		\begin{align*}\frac{w_2^2+w_3^2+w_4^2}{2}+w_3(w_2+w_4)&=w_1(w_2+w_3)+w_3w_4\\
		&=w_1(w_3+w_2+w_4)+w_2w_4\\
		&=w_1(w_4+w_3)+w_2w_3\end{align*}
		From the right hand sides, we deduce that \(w_2=w_4\) and \(w_3=w_1+w_2\). Now we solve the remaining equality:
		\[\frac{3w_2^2+w_1^2+2w_1w_2}{2}+2(w_1+w_2)w_2=w_1(w_1+2w_2)+(w_1+w_2)w_2.\]
		This ultimately reduces to \(w_1^2=5w_2^2\). This gives the following values of \(w_i\) and \(w(P)\):
		
		\[(w_1,w_2,w_3,w_4)=\frac{2\sqrt{5}-3}{11}(\sqrt5,1,1+\sqrt5,1),\]
		\[w(P)=\frac{1}{121}\left(5\sqrt5-2\right)=\frac16\alpha.\]
		
		This is greater than \(\frac2{27}\), so is optimal. Thus \(r!\lambda(P)=\alpha\), and \(w(v)>0\) as required by Theorem \ref{thm:nonjump}.
		
		It suffices to show that \(\FR_2(P)\) has Lagrangian \(\frac16\alpha\). Consider a maximal weighting \(w\) of \(\FR_2(P)\). Notice that the distribution of weight between \((2,2)\) and \((2,3)\) does not affect the weight of \(\FR_2(P)\setminus\{(2,1),(2,2),(2,3)\}\), so we may assume \(w((2,2))=w((2,3))\). To improve readability, we write \(a=w(1),b=w((2,2))+w((2,3)),c=w((2,1)),d=w(3),e=w(4)\). Now
		\[w\left(\FR_2(P)\right)=a\Bigg(\frac14b^2+bc+\frac12\left(c^2+d^2+e^2\right)+(b+c+e)d\Bigg)+(b+c)de+\frac14b^2c.\]
		It suffices to show this is at most \(\frac16\alpha\), subject to \(a+b+c+d+e=1\) and all of these non-negative. We condition on which weights are positive.
		
		\paragraph{Case 1: all \(w_i\) are positive.}
		
		Suppose \(a,b,c,d,e>0\). Now taking derivatives of \(w\left(\FR_2(P)\right)\), we find the following are all equal:
		\begin{align*}
		\frac{\partial w\left(\FR_2(P)\right)}{\partial a}&=\frac{b^2}{4}+bc+\frac{c^2+d^2+e^2}2+(b+c+e)d,\\
		\frac{\partial w\left(\FR_2(P)\right)}{\partial b}&= a\left(\frac{b}{2}+c+d\right)+de+\frac{bc}{2},\\
		\frac{\partial w\left(\FR_2(P)\right)}{\partial c}&= a(b+c+d)+de+\frac{b^2}{4},\\
		\frac{\partial w\left(\FR_2(P)\right)}{\partial d}&= a(b+c+d+e)+(b+c)e,\\
		\frac{\partial w\left(\FR_2(P)\right)}{\partial e}&= a(d+e)+(b+c)d.
		\end{align*}
		Comparing \(\frac{\partial w\left(\FR_2(P)\right)}{\partial b}\) and \(\frac{\partial w\left(\FR_2(P)\right)}{\partial c}\), we find \(c=a+\frac12b\). Comparing \(\frac{\partial w\left(\FR_2(P)\right)}{\partial d}\) and \(\frac{\partial w\left(\FR_2(P)\right)}{\partial e}\), we find \(d=a+e\). Subsituting into \(\frac{\partial w\left(\FR_2(P)\right)}{\partial c}\) and \(\frac{\partial w\left(\FR_2(P)\right)}{\partial d}\), we find
		\[a=\frac{b^2}{4e}-\frac32b+e.\]
		
		Now we have
		
		\begin{align*}0&=\frac{\partial w\left(\FR_2(P)\right)}{\partial a}-\frac{\partial w\left(\FR_2(P)\right)}{\partial b}=2e^2+\frac78b^2+\frac32ab\\
		&=e^2\left(\frac{3}{8}\left(\frac{b}{e}\right)^3-\frac{11}{8}\left(\frac{b}{e}\right)^2+\frac{3b}{2e}+2\right)\end{align*}
		
		But the cubic \(\frac38x^3-\frac{11}{8}x^2+\frac32x+2\) has no positive roots, so there are no solutions, and \(w\left(\FR_2(P)\right)\) has no local maximum with all weights positive. Thus it is sufficient to check that \(w\left(\FR_2(P)\right)\leq\frac16\alpha\) when some weight is zero.

		\paragraph{Case 2: some \(w_i\) is zero.} If \(a=0\), \(w\left(\FR_2(P)\right)=(b+c)de+\frac14b^2c\). Subject to a given value of \(x=b+c\), the first term is maximised when \(d=e\), and the second is maximised when \(b=2c\). Then 
		\[w\left(\FR_2(P)\right)=\frac14x(1-x)^2+\frac1{27}x^3.\]
		This is maximised when \(x=\frac{1}{31}(18-3\sqrt{5})\), in which case \(w\left(\FR_2(P)\right)\approx 0.0386<\frac16\alpha\).
		
		If \(b=0\) or \(c=0\), \(w\left(\FR_2(P)\right)\) is bounded above by the Lagrangian of \(G\), so is at most \(\frac16\alpha\).
		
		If \(d=0\), \[w\left(\FR_2(P)\right)=a\Bigg(\frac14b^2+bc+\frac12\left(c^2+e^2\right)\Bigg)+\frac14b^2c\]
		is maximised by some weight with \(e=0\), as otherwise we can take \(a'=a,b'=b,c'=c+e,e'=0\) without decreasing \(w\left(\FR_2(P)\right)\).
		
		If \(e=0\), \begin{align*}w\left(\FR_2(P)\right)&=a\Bigg(\frac14b^2+bc+\frac12\left(c^2+d^2\right)+(b+c)d\Bigg)+\frac14b^2c\\
		&=\frac12a(1-a)^2+\frac14b^2(c-a)\end{align*}
		
		This is clearly maximised at some weighting with \(d=0\) also, as changing to \(a'=a,b'=b,c'=c+d,d'=0\) again does not decrease \(w\left(\FR_2(P)\right)\).
		
		Thus we simply need to deal with the case \(a,b,c>0\), \(d=e=0\). But now we are maximising \(w\left(\FR_2(P)\right)=\frac12a(b+c)^2+\frac14b^2(c-a)\), so taking each derivative, we have
		\[\frac12(b+c)^2-\frac14b^2=a(b+c)+\frac12b(c-a)=a(b+c)+\frac14b^2.\]
		From this, we deduce \(c-a=\frac12b\), and thus \(\frac54b^2=a^2\). Hence this expression is maximised at
		\[a=\frac1{11}(10-3\sqrt5), b=\frac1{11}(4\sqrt5-6),c=\frac1{11}(7-\sqrt5).\]
		But these in fact give the bound \(w\left(\FR_2(P)\right)=\frac16\alpha\), as desired.
	\end{proof}

	By Theorem \ref{thm:increaser}, note that this also implies \(\frac{27r!}{121r^r}(5\sqrt5 - 2)\) is a jump for any \(r\geq 3\). However, for \(r\geq 4\), we may prove the following, smaller, non-jump.

	\begin{theorem}\label{thm:new2}
		For each \(r\geq 4\), the density \(\alpha=\frac{2r!}{r^r}\) is not a jump for \(r\)-graphs.
	\end{theorem}
	\begin{proof}
		By Theorem \ref{thm:increaser} it will be sufficient to show this for \(r=4\), where \(\alpha=\frac3{16}\) and \(\frac1{r!}\alpha=\frac1{128}\).

        We consider the \(4\)-pattern \(P\) with vertex set \([3]\) and a single edge, \(1233\). Now \(w(P)=\frac12 w(1) w(2) w(3)^2\) is maximised when \(w(1)=\frac12\), and \(w(2)=w(3)=\frac14\), giving \(\lambda(P)=\frac1{128}\). This assigns positive weight to \(3\), so it suffices to show \(\lambda\left(\FR_3(P)\right)=\frac1{128}\) also.

        Consider a maximal weighting \(w\) of \(\FR_3(P)\). We note that the expression \(w(\FR_3(P))\) is maximised when \(w(1)=w(2)\), and \(w((3,2))=w((3,3))=w((3,4))\). Again for notational convenience we write \(a=w(1)+w(2)\), \(b=w((3,2))+w((3,3))+w((3,4))\), and \(c=w((3,1))\). Now
		\[w(\FR_3(P))=\frac{a^2}{4}\left(bc+\frac13b^2+\frac12c^2\right)+\frac{b^3c}{27}.\]
		It suffices to prove that this is at most \(\frac1{128}\), subject to \(a+b+c=1\). If \(a=0\), this reduces to the last term, which is at most \(4^{-4}\). If \(b=0\) or \(c=0\), it reduces to an expression bounded above by something of the form \[2\left(\frac{a}{2}\right)^{2}\left(\frac{b}{2}\right)^2,\]
		which is easily seen to be at most \(\frac{1}{128}\). Thus we may assume \(a,b,c>0\). Now we know the following partial derivatives are all equal:
		\begin{align*}
		\frac{\partial w(\FR_3(P))}{\partial a}&=\frac{a}{2}\left(bc+\frac13b^2+\frac12c^2\right)\\
		\frac{\partial w(\FR_3(P))}{\partial b}&=\frac{a^2}{4}\left(c+\frac{2b}{3}\right)+\frac{b^2c}{9}\\
		\frac{\partial w(\FR_3(P))}{\partial c}&=\frac{a^2}{4}\left(b+c\right)+\frac{b^3}{27}.
		\end{align*}
		Equating the last two of these, we find
		\[c=\frac{3a^2}{4b}+\frac{b}{3}.\]
		Substituting in \(c\), we find \(\frac{\partial w(\FR_3(P))}{\partial a}=\frac{\partial w(\FR_3(P))}{\partial c}\) whenever, for \(x=\frac{a}{b}\),
		\[\frac9{64} x^5 - \frac{3}{16} x^4 + \frac12x^3 - \frac13x^2 + \frac{13}{36}x - \frac1{27}=0.\]
		Unfortunately this has a positive root, but only one, at \(x\approx 0.112324\). This corresponds to
		\[(a,b,c)\approx(0.0771923,0.687229,0.235579).\]
		This is a stationary point of \(w(\FR_3(P))\). Let \(f(b,c)\) be \(w(\FR_3(P))\) at the weighting given by \(a=1-b-c\). Then the Hessian matrix of \(f\) at this point is \[\begin{pmatrix}\frac{\partial^2 f}{\partial b^2}&\frac{\partial^2 f}{\partial b\partial c}\\\frac{\partial^2 f}{\partial b\partial c}&\frac{\partial^2 f}{\partial c^2}\end{pmatrix}\approx\begin{pmatrix}0.156956&0.16511\\0.16511&0.103793\end{pmatrix},\]
		which has determinant \(\approx -0.01097\). As this is negative, the second partial derivative test implies that this is a saddle point rather than a local maximum. Thus \(w(\FR_3(P))\) is maximised at a point where one of \(a,b,c\) is zero. Thus we can be sure that \(\lambda(\FR_3(P))\leq\frac1{128}\), and so \(\alpha\) is a jump, as desired.
		%TODO add plot
	\end{proof}

	Notice that this density is exactly double Erd\H{o}s' conjectured jump of \(\frac{r!}{r^r}\). We now show that these two jumps are the smallest that can be proven using the Frankl-R\"odl method.
	
	\begin{theorem}\label{thm:limitation}
		Suppose some \(r\)-pattern \(P\) meets the conditions of Theorem \ref{thm:nonjump} in order to prove \(\alpha\) is a non-jump for \(r\)-graphs. Then:
		\begin{enumerate}[label=(\roman*)]
			\item if \(r=3\), \(\alpha\) is at least \(\alpha^*_3=\frac6{121}(5\sqrt{5}-2)\), and
			\item for \(r\geq 4\), \(\alpha\) is at least \(\alpha^*_r=\frac{2r!}{r^r}\).
		\end{enumerate}
	\end{theorem}
	\begin{proof}
        Let \(v\) be as in the statement of Theorem \ref{thm:nonjump}, so that \(\lambda(P)=\lambda\left(\FR_v(P)\right)\). Suppose for contradiction that no edge of \(P\) contains \(v\) with multiplicity greater than one, and let \(w\) be a maximal weighting of \(P\). Now define a weighting \(w'\) of \(\FR_v(P)\) by \(w'((v,i))=\frac1{r}w(v)\) for each \(i\), and \(w'(u)=w(u)\) for each other \(u\in V(P)\). But then \(w'(\FR_v(P))=w(P)+\frac1{r^r}w(v)>\lambda(P)\), a contradiction. Hence some edge contains \(v\) with multiplicity at least two.
    
		(i): By the previous paragraph, together with the conditions of Theorem \ref{thm:nonjump}, some \(u\neq v\) has \(uvv\in P\). Consider a maximal weighting \(w\) of \(\FR_v(P)\). Noting again that we may assume \(w((v,2))=w((v,3))\), we write \(a=w(u)\), \(b=w((v,2))+w((v,3))\), and \(c=w((v,1))\). Now \(a+b+c=1\), and we have
		\[\lambda(\FR_v(P))=a\left(\frac14b^2+bc+\frac12c^2\right)+\frac14b^2c<\frac16\alpha.\]
		Now substituting in the values 
		\[a=\frac1{11}(10-3\sqrt5), b=\frac1{11}(4\sqrt5-6),c=\frac1{11}(7-\sqrt5),\]
		we find \(\alpha\geq \alpha_3^*\).
		
		(ii): Consider the edge \(e\) of \(P\) that contains \(v\) with multiplicity at least \(2\). Now consider the weighting \(w(u)=\frac1rm_e(u)\) for each vertex \(u\) of \(P\). Then
        \[\lambda(P)\geq w(P) =\frac1{r^r}\prod_{u\in V(P)}\frac{m_e(u)^{m_e(u)}}{m_e(u)!}\geq \frac1{r^r}\frac{m_e(v)^{m_e(v)}}{m_e(v)!}\geq \frac{2}{r^r}.\]
        Hence \(\alpha\geq2\frac{r!}{r^r}=\alpha_r^*\).
	\end{proof}

    We have shown that, for each \(r\geq 4\), the density \(2\cdot\frac{r!}{r^r}\) is a non-jump for \(r\)-graphs. When \(r=2\), this expression corresponds to the density \(1\), for which the notion of a jump is trivial as there are no graphs of higher density. This motivates the following question:

    \begin{question}
        Is the density \(\frac49=2\cdot\frac{3!}{3^3}\) a jump for \(3\)-graphs?
    \end{question}

    Erd\H{o}s' conjecture that \(\frac{r!}{r^r}\) should be a jump for \(r\)-graphs remains open. However, Baber and Talbot's result that \(\frac29\) cannot be a jump larger than \(10^{-4}\) for \(3\)-graphs suggests that this conjecture could be false for \(3\)-graphs. If so, it would fail for all \(r\geq 3\) by Theorem \ref{thm:increaser}. It would be extremely interesting to know whether \(\frac29\) is a jump for \(3\)-graphs, but our result shows it will require new techniques.

\bibliographystyle{plain}
\bibliography{hypergraph-jumps}

\end{document}